\documentclass[a4paper,11pt,reqno]{amsart}

\usepackage{amssymb,calc}
\usepackage{tikz}
\usepackage{epsfig}
\usepackage{amsfonts}
\usepackage{amsrefs}
\usepackage{amsmath}
\usepackage{euscript}
\usepackage{amscd}
\usepackage{amsthm}
\usepackage{enumitem}
\DeclareMathAlphabet{\mathpzc}{OT1}{pzc}{m}{it}
\usepackage{enumitem}
\usepackage[numbers,sort&compress]{natbib}
\usepackage{color}
\usepackage{mathtools}
\usepackage[hypertexnames=false, colorlinks, citecolor=red, linkcolor=blue, urlcolor=red]{hyperref}
\usetikzlibrary {intersections, calc, patterns, arrows}
\usepackage{marginnote}

\usepackage[normalem]{ulem}

\newcommand {\edge}[2] {\draw [->] (#1) -- ($ (#1) ! .5 ! (#2) $); \draw[] ($ (#1) ! .5 ! (#2) $) -- (#2); }
\newcommand {\smallloop}[1] {\draw [->] (#1) arc (270:90:.2); \draw (#1) arc (270: 450: .2);}
\newcommand {\bigloop}[1] {\draw [->] (#1) arc (270:90:.3); \draw (#1) arc (270: 450: .3);}
\newcommand {\labelpoint} [4] {\filldraw [fill = white] (#1) node[xshift = #2pt, yshift = #3pt] {$#4$} circle (1.5pt);}
\newcommand{\marginextend}[1]{ \addtolength{\oddsidemargin}{-#1}  \addtolength{\evensidemargin}{-#1}
	\addtolength{\textwidth}{#1}\addtolength{\textwidth}{#1}}
\newcommand{\updownextend}[1]{ \addtolength{\topmargin}{-#1}  \addtolength{\textheight}{#1}
	\addtolength{\textheight}{#1}}
\marginextend{1.5cm}
\updownextend{0cm}
\allowdisplaybreaks[4]

\usepackage{pst-node}
\usepackage{tikz-cd}
\usepackage{mathrsfs}
\DeclareFontFamily{OT1}{pzc}{}
\DeclareFontShape{OT1}{pzc}{m}{it}{<-> s * [1.10] pzcmi7t}{}
\DeclareMathAlphabet{\mathpzc}{OT1}{pzc}{m}{it}

\DeclareSymbolFont{SY}{U}{psy}{m}{n}
\DeclareMathSymbol{\emptyset}{\mathord}{SY}{'306}

\theoremstyle{plain}

\newtheorem{thm}{Theorem}[section]
\newtheorem*{thm*}{Theorem}
\newtheorem{cor}[thm]{Corollary}
\newtheorem{lem}[thm]{Lemma}

\newtheorem{defn}[thm]{Definition}
\newtheorem{rem}[thm]{Remark}

\newtheoremstyle{named}{}{}{\itshape}{}{\bfseries}{.}{.5em}{#1 \thmnote{#3}}
\theoremstyle{named}

\numberwithin{equation}{section}

\def\beq{\begin{eqnarray}}
	\def\eeq{\end{eqnarray}}
\def\beqa{\begin{eqnarray*}}
	\def\eeqa{\end{eqnarray*}}

\newcommand{\be}{\begin{equation}}
	\newcommand{\ee}{\end{equation}}
\newcommand{\bea}{\begin{eqnarray}}
	\newcommand{\eea}{\end{eqnarray}}
\newcommand{\Bea}{\begin{eqnarray*}}
	\newcommand{\Eea}{\end{eqnarray*}}

\newcounter{cnt1}
\newcounter{cnt2}
\newcounter{cnt3}
\newcommand{\blr}{\begin{list}{$($\roman{cnt1}$)$}
		{\usecounter{cnt1} \setlength{\topsep}{0pt}
			\setlength{\itemsep}{0pt}}}
	\newcommand{\bla}{\begin{list}{$($\alph{cnt2}$)$}
			{\usecounter{cnt2} \setlength{\topsep}{0pt}
				\setlength{\itemsep}{0pt}}}
		\newcommand{\bln}{\begin{list}{$($\arabic{cnt3}$)$}
				{\usecounter{cnt3} \setlength{\topsep}{0pt}
					\setlength{\itemsep}{0pt}}}
			\newcommand{\el}{\end{list}}
		
		\vspace{5mm}

		\title{On outer automorphisms of certain graph $\textrm{C}^{\ast}$-algebras}
		\author[S. Datta]{Swarnendu Datta}
		
		\author[D. Goswami]{Debashish Goswami}
		
		\author[S. Joardar]{Soumalya Joardar}

        \address[S. Datta]{Indian Institute of Science Education And Research Kolkata, Mohanpur 741246, Nadia, West Bengal, India} \email{swarnendu.datta@iiserkol.ac.in}
        
        \address[D. Goswami]{Indian Statistical Institute, 203, B.T. Road, Kolkata-700108, India} \email{debashish\textunderscore goswami@yahoo.co.in}

            \address[S. Joardar]{Indian Institute of Science Education And Research Kolkata, Mohanpur 741246, Nadia, West Bengal, India} \email{soumalya@iiserkol.ac.in}

\begin{document}
			\begin{abstract}	   
		Given a countable abelian group $A$, we construct a row finite directed graph $\Gamma(A)$ such that the $K_{0}$-group of the graph $\textrm{C}^{\ast}$-algebra $\textrm{C}^{\ast}(\Gamma(A))$ is canonically isomorphic to $A$. Moreover, each element of $\textrm {Aut}(A)$ is a lift of an automorphism of the graph $\textrm{C}^{\ast}$-algebra $\textrm{C}^{\ast}(\Gamma(A))$.
			\end{abstract}
			\maketitle
            \section{Introduction}
            Given a $\textrm{C}^{\ast}$-algebra $\mathcal{C}$, it is an important problem to understand its automorphism group $\textrm {Aut}(\mathcal{C})$. The automorphism group sometimes encodes important structural information of the $\textrm{C}^{\ast}$-algebra. It also helps to construct new $\textrm{C}^{\ast}$-algebras. The reader is referred to the book by Pedersen (\cite{pederson}) for generalities of automorphism groups of $\textrm{C}^{\ast}$-algebras. Often it is difficult to understand the full automorphism group of a $\textrm{C}^{*}$-algebra. The usefulness of $K$-theory in understanding $\textrm{C}^{\ast}$-algebras is now well documented (see for example \cite{elliot,phillips}). It turns out that it also helps to understand the automorphism groups of $\textrm{C}^{\ast}$-algebras. Given a $\textrm{C}^{*}$-algebra $\mathcal{C}$,  one important normal subgroup of $\textrm {Aut}(\mathcal{C})$ is the subgroup of its inner automorphism group to be denoted by $\textrm {Inn}(\mathcal{C})$. These are of the form $c\rightarrow ucu^{\ast}$ for some unitary $u$ in the unitization of $\mathcal{C}$. One key tool to understand the automorphism group is the induced automorphism of the abelian group $K_{0}(\mathcal{C})$. Recall that by the functorial property of $K_{0}$, for any $\phi\in{\rm Aut}(\mathcal{C})$, $K_{0}(\phi)\in {\rm Aut}(K_{0}(\mathcal{C}))$. It turns out that any inner automorphism induces the trivial automorphism of the $K_{0}$-group. In fact, something more is true. There is a larger normal subgroup known as the approximately inner automorphism group (to be denoted by $\overline{\rm Inn}(\mathcal{C})$) such that each element of $\overline{\rm Inn}(\mathcal{C})$ induces the trivial automorphism on the $K_{0}$-group (see \cite{rordam}). Thus, the $K$-outer automorphism group of a $\textrm{C}^{\ast}$-algebra (to be denoted by $\textrm {Kout}(\mathcal{C})$) is naturally defined as the quotient group $\textrm {Aut}(\mathcal{C})/\overline{\textrm {Inn}}(\mathcal{C})$. Then it becomes important to understand the $K$-outer automorphisms of $\textrm{C}^{\ast}$-algebras. It is clear that if any non-trivial automorphism of $K_{0}(\mathcal{C})$ is induced by some $\phi\in{\rm Aut}(\mathcal{C})$, then $\phi$ has to be $K$-outer. We call an automorphism of the $K_{0}$ group of the form $K_{0}(\phi)$ for some $\phi\in{\rm Aut}(\mathcal{C})$ a lift. It is easy to see that all elements of ${\rm Aut}(K_{0}(\mathcal{C}))$ need not be lifts in general. For example, $K_{0}(M_{n}(\mathbb{C}))=\mathbb{Z}$. Hence ${\rm Aut}(K_{0}(M_{n}(\mathbb{C})))=\mathbb{Z}_{2}$. It is well known that all automorphisms of $M_{n}(\mathbb{C})$ are inner and, therefore, the non-trivial automorphism of $\mathbb{Z}$ is not a lift. However, there are classes of $\textrm{C}^{\ast}$-algebras whose $K$-outer automorphism groups are well understood. For example, if $\mathcal{C}$ is an AF-algebra, then one has the following short exact sequence
           \begin{displaymath}
           0\longrightarrow\overline{\rm Inn}(\mathcal{C})\hookrightarrow{\rm Aut}(\mathcal{C})\longrightarrow{\rm Aut}^{+}(K_{0}(\mathcal{C}))\longrightarrow 0,    
           \end{displaymath}
           where $\textrm {Aut}^{+}(K_{0}(\mathcal{C}))$ is the group of automorphisms of the $K_{0}$ group preserving an extra order structure. Consequently, the $K$-outer automorphism group is isomorphic to $\textrm {Aut}^{+}(K_{0}(\mathcal{C}))$ (see Exercise (7.8) of \cite{rordam}). \vspace{0.05in}\\
           \indent Keeping this context in mind, in this article, we provide a large class of $\textrm{C}^{\ast}$-algebras such that each element of the automorphism group of the $K_{0}$-group is a lift. It enables us to construct large $K$-outer automorphisms. The class is constructed out of graph $\textrm{C}^{\ast}$-algebras. Recall that given a row-finite directed graph $G$, one associates a $\textrm{C}^{\ast}$-algebra $\textrm{C}^{\ast}(G)$. One of many benefits of building $\textrm{C}^{\ast}$-algebras from directed graphs is that a lot of structural information can be obtained from the combinatorial structure of the underlying graphs. For example, the $K_{0}$-group is given by the cokernel of a linear map associated to the adjacency matrix of the graph (see \cite{raeburn}). In addition, elements of the graph automorphism group produce automorphisms of the graph $\textrm{C}^{\ast}$-algebra. For a finite, directed graph such automorphisms and their quantum version have been studied in \cite{joardar}. In this article, given any countable abelian group $A$, we construct a row-finite directed graph $\Gamma(A)$ such that $K_{0}(\textrm{C}^{\ast}(\Gamma(A)))\cong A$ in such a way that every automorphism of $A$ is a lift. Note that given a countable abelian group $A$, a graph $\textrm{C}^{\ast}$-algebra $\mathcal{C}$ with $K_{0}(\mathcal{C})\cong A$ has been constructed in \cite{sz}. In fact in \cite{sz}, given a pair of abelian groups $(A_{0},A_{1})$ where $A_{0}$ is countable and $A_{1}$ free, a graph $\textrm{C}^{\ast}$-algebra $\mathcal{C}$ has been constructed such that $K_{i}(\mathcal{C})\cong A_{i}$. However, the novelty of our construction is that every automorphism of the given $K_{0}$-group is a lift which, in general, need not be true as mentioned earlier. As a consequence, we show that the automorphism group of the $K_{0}$-group of the graph $\textrm{C}^{\ast}$-algebra is a subgroup of the $K$-outer automorphism group of the graph $\textrm{C}^{\ast}$-algebra. It is worth mentioning that recently a quantum version of outer automorphism groups has been formulated and studied for von Neumann algebras with tracial states (\cite{goswami}). We hope that the present article could act as a stepping stone in understanding the quantum version of outer automorphism groups of graph $\textrm{C}^{\ast}$-algebras in the long run.   

            \section{Preliminaries}
            \subsection{Graph $\textrm{C}^{\ast}$-algebras and their $K_{0}$-groups}
            We begin this subsection by discussing the rudiments of graph $\textrm{C}^{*}$-algebras. The reader is referred to Chapter 1 of \cite{raeburn} for details. Recall that a directed graph $G$ is a collection $(V, E, r,s)$ where $V$ is a set consisting of countably many points known as the set of vertices; $E$ is another countable set known as the set of edges; $r,s:E\rightarrow V$ are maps known as the range and source maps. For an edge $e\in E$ such that $s(e)=v\in V$ and $r(e)=w\in V$, we often write $e$ as $(v,w)$. A graph is called {\bf row-finite} if $r^{-1}(v)$ is finite for all $v\in V$. We shall only consider {\it row-finite directed} graph in this paper. Given a row-finite directed graph $G=(V,E,r,s)$, a Cuntz-Krieger family is a collection $\{\{p_{v}\}_{v\in V},\{S_{e}\}_{e\in E}\}$ where $\{p_{v}\}_{v\in V}$ are mutually orthogonal projections and $\{S_{e}\}_{e\in E}$ are partial isometries satisfying the following relations:\vspace{0.1in}\\
            (CK1) $S_{e}^{\ast}S_{e}=p_{s(e)}$\\
            (CK2) $\sum_{r(e)=v}{S_{e}S_{e}^{\ast}}=p_{v}$.
            \begin{defn}
                Let $G$ be a row-finite directed graph. Then the graph $C^{\ast}$-algebra $C^{\ast}(G)$ is defined to be the universal $C^{\ast}$-algebra generated by the Cuntz-Krieger families.
            \end{defn}
           Now let us briefly recall the $K_{0}$-group of $\textrm{C}^{\ast}(G)$ for a row-finite directed graph $G$. For details, the reader is referred to \cite{szymanski} or \cite{raeburn}. To that end let us denote the set $\{v\in V: r^{-1}(v)\neq\emptyset\}$ by $V_{E}$. We denote the free abelian groups on the sets $V_{E}$ and $V$ by $\mathbb{Z}V_{E}$ and $\mathbb{Z}V$ respectively as usual. Then define a map $B$ on an element $v\in V_{E}$ by $B(v)=v-\sum_{r(e)=v}s(e)\in\mathbb{Z}V$. Extending this map $\mathbb{Z}$-linearly on the whole of $\mathbb{Z}V_{E}$, we get a group homomorphism $B:\mathbb{Z}V_{E}\rightarrow\mathbb{Z}V$. Then the $K_{0}$ group is isomorphic to the cokernel of $B$ (see Proposition 2 of \cite{szymanski}).
           \begin{rem}
               1. Note that in terms of projections, the abelian group $K_{0}(C^{\ast}(G))$ is generated by the classes $\{[p_{v}]\}_{v\in V}$. We get the isomorphism $K_{0}(C^{\ast}(G))\cong \textrm {coker}(B)$ by mapping $[p_{v}]$ to $v$.\vspace{0.05in}\\
               \indent 2. It is also important to recall the functor $K_{0}$. Given a $C^{*}$-homomorphism $\phi:\mathcal{C}\rightarrow \mathcal{D}$ between two $C^{\ast}$-algebras, the map $K_{0}(\phi):K_{0}(\mathcal{C})\rightarrow K_{0}(\mathcal{D})$ sends a class of projection $p\in \mathcal{C}$ in the $K_{0}$-group of $\mathcal{C}$ to the class of projection $\phi(p)\in \mathcal{D}$ in $K_{0}(\mathcal{D})$ (see \cite{rordam}).
           \end{rem}
            \subsection{Automorphisms of graphs and their $\textrm{C}^{\ast}$-algebras} We continue to work with a row-finite directed graph $G=(V,E,r,s)$. \begin{defn}
                An automorphism of a row-finite directed graph $G=(V,E,r,s)$ is a bijection $\phi:V\rightarrow V$ such that $(v,w)\in E$ if and only if $(\phi(v),\phi(w))\in E$.
            \end{defn}
            \begin{rem}
                We write the set of automorphisms of a row-finite directed graph $G$ by $\textrm {Aut}(G)$. Note that if $\phi$ is an automorphism of a graph, then it is also a bijection of the edge set. We write the image of an edge $e\in E$ under $\phi$ naturally by $\phi(e)$.
            \end{rem}
            For any $\phi\in \textrm {Aut}(G)$, there is an induced automorphism $\widetilde{\phi}$ of the graph $\textrm{C}^{\ast}$-algebra given on the generating projections and partial isometries, respectively, by 
            \begin{displaymath}
                \widetilde{\phi}(p_{v}):=p_{\phi(v)}, \ \widetilde{\phi}(S_{e})=S_{\phi(e)}.
            \end{displaymath}
            Indeed, it is straightforward to check that $\{\{\widetilde{\phi}(p_{v})\}_{v\in V}, \  \{\widetilde{\phi}(S_{e})\}_{e\in E}\}$ is again a Cuntz-Krieger family, that is, they satisfy the relations (CK1) and (CK2). Therefore, by the universal property, there is a well-defined $\textrm{C}^{\ast}$-homomorphism $\widetilde{\phi}:\textrm{C}^{\ast}(G)\rightarrow \textrm{C}^{\ast}(G)$. Using $\phi^{-1}$, one can similarly define $\widetilde{\phi^{-1}}:\textrm{C}^{\ast}(G)\rightarrow \textrm{C}^{\ast}(G)$ and it is easy to see that $\widetilde{\phi^{-1}}=(\widetilde{\phi})^{-1}$ so that $\widetilde{\phi}\in\textrm {Aut}(\textrm{C}^{\ast}(G))$.\\
            \indent Given a row-finite directed graph $G$, recall the $K_{0}$-group of the $\textrm{C}^{\ast}$-algebra $\textrm{C}^{\ast}(G)$ from the previous subsection. 
            \begin{lem}
              An automorphism $\phi$ of a row finite directed graph $G$ descends to an automorphism of the abelian group $K_{0}(C^{\ast}(G))$. We denote the automorphism of $K_{0}(C^{\ast}(G))$ corresponding to a graph automorphism $\phi$ by $\underline{\phi}$.
            \end{lem}
            \begin{proof}
            Recall that $K_{0}(\textrm{C}^{\ast}(G))$ is isomorphic to the cokernel of the group homomorphism $B:\mathbb{Z}V_{E}\rightarrow\mathbb{Z}V$ given on the $\mathbb{Z}$-linear basis $\{v\}_{v\in V_{E}}$ by
            \begin{displaymath}
                B(v)=v-\sum_{e:r(e)=v}s(e).
            \end{displaymath}
            As $\phi\in\textrm {Aut}(G)$, it maps $V_{E}$ to $V_{E}$. Then extending $\phi$ $\mathbb{Z}$-linearly, we get maps from $\mathbb{Z}V_{E}$ to $\mathbb{Z}V_{E}$ and $\mathbb{Z}V$ to $\mathbb{Z}V$. We continue to denote the extensions by $\phi$. For $v\in V_{E}$,
            \begin{displaymath}
                B(\phi(v))=\phi(v)-\sum_{r(e)=\phi(v)}s(e).
            \end{displaymath}
            As $\phi$ is an automorphism of the graph $G$, for any $e\in E$ such that $r(e)=\phi(v)$, $s(e)=\phi(w)$ for some $w\in E$ such that $e^{\prime}=(w,v)\in E$ and  $\phi(e^{\prime})=e$. Consequently, 
            \begin{displaymath}
                B(\phi(v))=\phi\Big(v-\sum_{r(e^{\prime})=v}s(e^{\prime})\Big)=\phi(Bv).
            \end{displaymath}
            Therefore, by the $\mathbb{Z}$-linearity of the maps $\phi$ and $B$, we get the following commutative diagram:
            \[
\begin{tikzcd}
\mathbb{Z}V_{E} \arrow{r}{B} \arrow[swap]{d}{\phi} & \mathbb{Z}V \arrow{d}{\phi} \\
\mathbb{Z}V_{E} \arrow{r}{B} & \mathbb{Z}V
\end{tikzcd}
\]
Hence we get a well-defined group homomorphism $\underline{\phi}:{\rm coker}(B)\rightarrow {\rm coker}(B)$ and consequently a group homomorphism $\underline{\phi}:K_{0}(\textrm{C}^{\ast}(G))\rightarrow K_{0}(\textrm{C}^{\ast}(G))$. Repeating the argument with $\phi^{-1}$, we get a group homomorphism $\underline{\phi}^{-1}:K_{0}(\textrm{C}^{\ast}(G))\rightarrow K_{0}(\textrm{C}^{\ast}(G))$. It is straightforward to verify $(\underline{\phi})^{-1}=\underline{\phi}^{-1}$ proving that $\underline{\phi}\in\textrm {Aut}\Big(K_{0}(\textrm{C}^{\ast}(G))\Big)$.
            \end{proof}
            
            \begin{lem}
            \label{lift}
                Let $\phi\in\textrm {Aut}(G)$, where $G$ is a row-finite directed graph. Then the induced automorphism $\underline{\phi}\in\textrm {Aut}\Big(K_{0}(C^{\ast}(G))\Big)$ is a lift.
            \end{lem}
            \begin{proof}
             We shall prove that $\underline{\phi}=K_{0}(\widetilde{\phi})$. But this is more or less straightforward once we note that the $K_{0}$-group of $\textrm{C}^{\ast}(G)$ is generated by the class of projections $\{p_{v}\}_{v\in V}$ in $K_{0}$-group and the action of $K_{0}(\widetilde{\phi})$ on a class $[p_{v}]$ is given by $K_{0}(\widetilde{\phi})([p_{v}])=[p_{\phi(v)}]$. Then identifying $v\in V$ with $[p_{v}]$, we see that the actions of $K_{0}(\widetilde{\phi})$ and $\underline{\phi}$ agree on the generators and hence they agree on $K_{0}(\textrm{C}^{\ast}(G))$.    
            \end{proof}

           \section{Main section}
           In this section, given a countable abelian group $A$, we construct a row finite directed graph $\Gamma(A)$ such that 
           \begin{itemize}
               \item (MC1) $K_{0}(\textrm{C}^{\ast}(\Gamma(A)))\cong A$.
               \item (MC2) Any $\phi\in{\rm Aut}(A)$ induces an automorphism $\Gamma(\phi)\in{\rm Aut}(\Gamma(A))$ such that $\phi=\underline{\Gamma(\phi)}\in{\rm Aut}(A)$.  
                        \end{itemize}
                         We shall achieve the above by a functorial construction. To that end, let $\mathcal{A}$ be the category whose objects are countable abelian groups and morphisms are group homomorphisms. Let $\mathcal{G}$ be the category whose objects are row-finite directed graphs without multiple edges admitting possibly loops such that a vertex can be a base to finitely many loops. The morphisms of the category $\mathcal{G}$ are graph homomorphisms. For completeness, let us recall the notion of graph homomorphism.
                        \begin{defn}
                            If $G_{1},G_{2}$ are two graphs in $\mathcal{G}$, then a homomorphism $\alpha:G_{1}\rightarrow G_{2}$ is a map from the vertices $V(G_{1})$ to $V(G_{2})$ such that whenever there is an edge $e$ in $G_{1}$ with $s(e)=v$, $r(e)=w$ one has corresponding edge $f$ in $G_{2}$ such that $r(f)=\alpha(w)$ and $s(f)=\alpha(v)$.
                        \end{defn}
                        Now we shall construct a functor $\Gamma:\mathcal{A}\rightarrow\mathcal{G}$ in such a way that $K_{0}\Big(\textrm{C}^{\ast}(\Gamma(A))\Big)$ is canonically isomorphic to $A$ for some abelian group $A\in\mathcal{A}$.\vspace{0.2in}\\
                        {\bf The main construction}: We first note that $A$ has a canonical presentation obtained as follows: let $F(A)$ denote the free abelian group generated by elements $a\in A$; there is a natural surjrctive group homomorphism:
                        \begin{align*}
                          &F(A)\rightarrow A\\
                          & \{a\}\mapsto a,
                        \end{align*}
                            whose kernel is generated by the elements
                            \begin{displaymath}
                                \{a\}+\{b\}-\{c\}
                            \end{displaymath}
                            whenever $a+b=c$ in $A$. Thus a presentation of $A$ is given by:\\
                            (i) generators: $\{a\}$, $a\in A$\\
                            (ii) relations: $\{a\}+\{b\}-\{c\}=0$, whenever $a+b=c$ in $A$.\\
                            Now we build the graph $\Gamma(A)$ as follows: for each $a\in A$, we add a vertex $v_{a}$. If there are no edges between $v_{a}$'s, the $K_{0}$ group of the graph $\textrm{C}^{\ast}$-algebra $\textrm{C}^{\ast}(\Gamma(A))$ is a free abelian group with generators $v_{a}$. Here with an abuse of notation, we identify a vertex with the corresponding generator in the $K_{0}$ group. Now for $a+b=c$ in $A$ we want the corresponding relation $v_{a}+v_b=v_{c}$ in the $K_{0}$ group. To this end, we add auxilliary vertices $u_{ab}$ and $u_{c}$ with edges as shown below:
                            
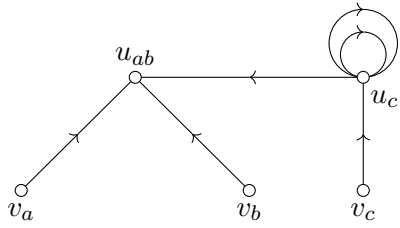
\begin {center}
\begin {tikzpicture}[scale = 1.5]

\coordinate (p) at (0,0);
\coordinate (q) at (1,1);
\coordinate (r) at (2,0);
\coordinate (s) at (3,1);
\coordinate (t) at (3,0);

\edge{p}{q};
\edge {r}{q};
\edge {s}{q};
\edge {t}{s};

\smallloop {s};
\bigloop {s};

\labelpoint {p}{0}{-8}{v_a}
\labelpoint {r}{0}{-8}{v_b}
\labelpoint {t}{0}{-8}{v_c}
\labelpoint {q}{0}{8}{u_{ab}}
\labelpoint {s}{8}{-8}{u_{c}}
\end {tikzpicture}
\end {center}  

                            Then in the $K_{0}$ group, we have the following relations:
                            \begin{eqnarray}
                              u_{ab}=v_{a}+v_{b}+u_{c}\\
                                u_{c}=v_{c}+2u_{c}  
                            \end{eqnarray}
                            Combining the above two equations, we get $u_{ab}=v_{a}+v_{b}-v_{c}$. It remains to add the relation $u_{ab}=0$. This can be achieved as follows: observe that if we add one more auxilliary vertex $u_{ab}^{1}$ connected to $u_{ab}$ as below: 
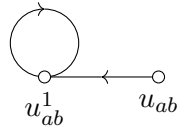
\begin {center}
\begin {tikzpicture}[scale = 1.5]

\coordinate (p) at (0,0);
\coordinate (q) at (1,0);

\edge{q}{p};
\bigloop {p}
\labelpoint {p}{0}{-12}{u^1_{ab}}
\labelpoint {q}{0}{-10}{u_{ab}}
\end {tikzpicture}
\end {center}
\vspace {20pt}
                            then we get the relation
                            \begin{displaymath}
                                u^{1}_{ab}=u_{ab}^{1}+u_{ab},
                            \end{displaymath}
                            i.e. $u_{ab}=0$. To set $u_{ab}^{1}=0$, we inductively add vertices $u^{n}_{ab}$ for $n\geq 2$ as follows:
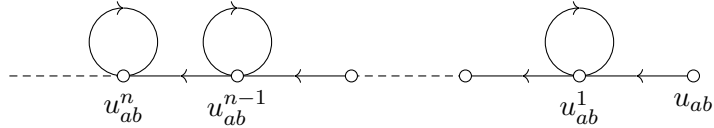
\begin {center}
\begin {tikzpicture}[scale = 1.5]

\coordinate (m) at (-1,0);
\coordinate (p) at (0,0);
\coordinate (q) at (1,0);
\coordinate (r) at (2,0);
\coordinate (s) at (3,0);
\coordinate (t) at (4,0);
\coordinate (u) at (5,0);

\edge{q}{p}
\edge {r}{q}
\edge {t}{s}
\edge {u}{t}

\bigloop {p}
\bigloop {q}
\bigloop {t}

\draw [densely dashed] (r) -- (s);
\draw [densely dashed] (m) -- (p);

\labelpoint {p}{0}{-12}{u^n_{ab}}
\labelpoint {q}{0}{-12}{u^{n-1}_{ab}}
\labelpoint {r}{0}{0}{};
\labelpoint {s}{0}{0}{};
\labelpoint {t}{0}{-12}{u^{1}_{ab}}
\labelpoint {u}{0}{-10}{u_{ab}}

\end {tikzpicture}
\end {center}
\vspace {20pt}

                            which gives as before 
                            \begin{displaymath}
                                u^{n}_{ab}=u^{n-1}_{ab}+u^{n}_{ab},
                            \end{displaymath}
                            i.e. $u_{ab}^{n}=0$ for all $n$ in the $K_{0}$-group. So we have been able to add the relation $v_{a}+v_{b}=v_{c}$ in the $K_{0}$-group. We repeat the construction for each triple $\{a,b,c\}$ in $A$ satisfying $a+b=c$.\\
                            \indent However, if $a=b\in A$, then $v_{a}=v_{b}$ which gives two edges between $u_{ab}$ and $v_{a}=v_{b}$. But this is not allowed and we solve this problem by always adding two intermediate vertices $w_{a}, w_{b}$ in the following way:
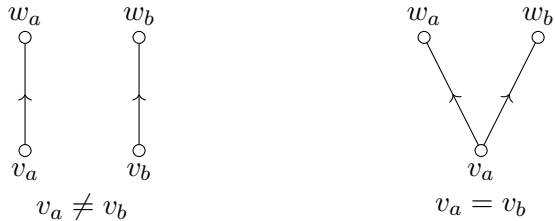
\begin {center}
\begin {tikzpicture}[scale = 1.5]

\coordinate (p) at (0,0);
\coordinate (q) at (0,1);
\coordinate (r) at (1,0);
\coordinate (s) at (1,1);

\edge {p}{q}
\edge {r}{s}

\labelpoint {p}{0}{-8}{v_a}
\labelpoint {q}{0}{8}{w_a}
\labelpoint {r}{0}{-8}{v_b}
\labelpoint {s}{0}{8}{w_b}

\draw (.5,-.5) node[] {$v_a \neq v_b$};

\coordinate (p) at (4,0);
\coordinate (q) at (3.5,1);
\coordinate (r) at (4.5,1);

\edge {p}{q}
\edge {p}{r}

\labelpoint {p}{0}{-8}{v_a}
\labelpoint {q}{0}{8}{w_a}
\labelpoint {r}{0}{8}{w_b}

\draw (4,-.5) node[] {$v_a = v_b$};

\end {tikzpicture}
\end {center}

                            Note that even if $v_{a}=v_{b}$, the vertices $w_{a}, w_{b}$ are distinct and we have $w_{a}=v_{a}$ and $w_{b}=v_{b}$ in the $K_{0}$-group. Now the construction can be repeated by adding $u_{ab}$ to $w_{a}, w_{b}$ instead of $v_{a},v_{b}$ as indicated below:
                            
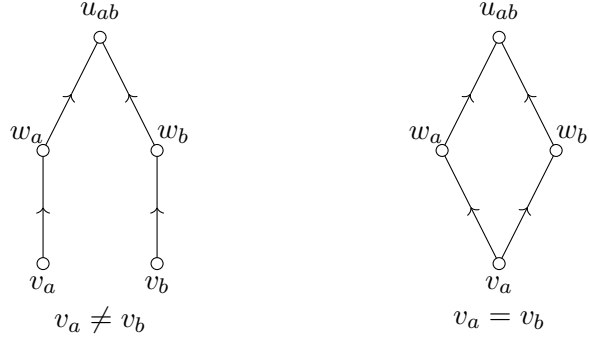
\begin {center}
\begin {tikzpicture}[scale = 1.5]

\coordinate (p) at (0,0);
\coordinate (q) at (0,1);
\coordinate (r) at (1,0);
\coordinate (s) at (1,1);
\coordinate (t) at (.5, 2);

\edge {p}{q}
\edge {r}{s}
\edge {q}{t}
\edge {s}{t}

\labelpoint {p}{0}{-8}{v_a}
\labelpoint {q}{-6}{6}{w_a}
\labelpoint {r}{0}{-8}{v_b}
\labelpoint {s}{6}{6}{w_b}
\labelpoint {t}{0}{10}{u_{ab}}

\draw (.5,-.5) node[] {$v_a \neq v_b$};

\coordinate (p) at (4,0);
\coordinate (q) at (3.5,1);
\coordinate (r) at (4.5,1);
\coordinate (s) at (4,2);

\edge {p}{q}
\edge {p}{r}
\edge {q}{s}
\edge {r}{s}

\labelpoint {p}{0}{-8}{v_a}
\labelpoint {q}{-6}{6}{w_a}
\labelpoint {r}{6}{6}{w_b}
\labelpoint {s}{0}{10}{u_{ab}}

\draw (4,-.5) node[] {$v_a = v_b$};

\end {tikzpicture}
\end {center}
                            The graph obtained in this way is denoted naturally by $\Gamma(A)$. From the construction, it is clear that $\Gamma(A)$ belongs to the category $\mathcal{G}$ and $K_{0}\Big(\textrm{C}^{\ast}(\Gamma(A))\Big)$ is isomorphic to $A$. The isomorphism is obtained by sending $v_{a}$ to $a$. To finish the construction of the functor $\Gamma$, given a group homomorphism $\phi:A\rightarrow B$,  we need to assign a graph homomorphism $\Gamma(\phi):\Gamma(A)\rightarrow\Gamma(B)$. But this is more or less obvious. In deed, given $\phi:A\rightarrow B$, $\Gamma(\phi)$ can be defined by sending $v_{a}$ to $v_{\phi(a)}$. Extension to the auxilliary verices is obvious as $a+b=c$ implies $\phi(a)+\phi(b)=\phi(c)$. It is clear from the construction that $\Gamma(\phi)$ is a graph homomorphism, $\Gamma(\phi\circ\psi)=\Gamma(\phi)\circ\Gamma(\psi)$ for group homomorphisms
                            \begin{displaymath}
                                A\longrightarrow^{\psi} B\longrightarrow^{\phi} C,
                            \end{displaymath}
                            and $\Gamma(\textrm{id}_{A})=\textrm{id}_{\Gamma(A)}$. Therefore, for any $\phi\in\textrm{Aut}(A)$, $\Gamma(\phi)\in\textrm{Aut}(\Gamma(A))$.
                            \begin{lem}
                                Let $\phi\in\textrm{Aut}(A)$ for a countable abelian group $A$. Then $\underline{\Gamma(\phi)}=\phi$.
                            \end{lem}
                            \begin{proof}
                                This follows essentially from the construction. Note that the $K_{0}$ group of $\textrm{C}^{\ast}(\Gamma(A))$ is generated by the elements $\{v_{a}, a\in A\}$ and the isomorphism with $A$ is obtained by sending $v_{a}$ to $a\in A$ canonically. Then the action of $\Gamma(\phi)$ on $v_{a}$ is by definition $v_{\phi(a)}$ which by the identification of $K_{0}$ group with $A$ is nothing but the map $\phi$.
                            \end{proof}
                            This completes our main construction satisfying (MC1) and (MC2).
                        \begin{cor}
                            Every $\phi\in{\textrm {Aut}}\Big(K_{0}(C^{\ast}(\Gamma(A)))\Big)$ is a lift. 
                        \end{cor}
                        \begin{proof}
                            By construction, $\phi=\underline{\Gamma(\phi)}$ for $\Gamma(\phi)\in{\rm Aut}(\Gamma(A))$. By Lemma \ref{lift}, $\underline{\Gamma(\phi)}$ and consequently $\phi$ is a lift. 
                        \end{proof}
                        \begin{cor}
                            ${\rm Aut}(A)$ is a subgroup of the $K$-outer automorphism group of the $C^{\ast}$-algebra $C^{\ast}(\Gamma(A))$.
                        \end{cor}
                        \begin{proof}
                            Given $\phi\in \textrm {Aut}(A)$, by the main construction, we have an element $\Gamma(\phi)\in\textrm {Aut}(\Gamma(A))$ such that $\underline{\Gamma(\phi)}=\phi$. Then we define a map $\beta:\textrm {Aut}(A)\rightarrow {\rm Kout}\Big(\textrm{C}^{\ast}(\Gamma(A))\Big)$ by
                            \begin{displaymath}
                                \beta(\phi)=[\widetilde{\Gamma(\phi)}],
                            \end{displaymath}
                            where $[\widetilde{\Gamma(\phi)}]$ is the class of $\widetilde{\Gamma(\phi)}\in{\rm Aut}\Big(\textrm{C}^{\ast}(\Gamma(A))\Big)$ in the group ${\rm Kout}\Big(\textrm{C}^{\ast}(\Gamma(A))\Big)$. By the main construction, $\beta$ is a group homomorphism. If $\beta(\phi)$ is the trivial element for some $\phi$, then $\widetilde{\Gamma(\phi)}\in\overline{\rm Inn}(\textrm{C}^{\ast}(\Gamma(A)))$ and therefore $K_{0}(\widetilde{\Gamma(\phi)})$ is the trivial automorphism of $K_{0}(\textrm{C}^{\ast}(\Gamma(A)))$. But by Lemma \ref{lift}, $K_{0}(\widetilde{\Gamma(\phi)})=\underline{\Gamma(\phi)}$ which, by construction, is $\phi$. Therefore, $\beta$ is injective, identifying $\textrm {Aut}(A)$ as a subgroup of $\textrm {Kout}\Big(\textrm{C}^{\ast}(\Gamma(A))\Big)$.
                        \end{proof}
                        \begin{rem}
                            It is clear from the construction that the graphs admit multiple loops so that the corresponding graph $C^{\ast}$-algebras can not be AF. This can also be seen from the $K$-outer automorphism groups. Recall that the $K$-outer automomrphism group of an AF algebra is isomorphic to the positive group isomorphism of the $K_{0}$-group whereas here the whole automorphism group of the $K_{0}$-group is a subgroup of the $K$-outer automorphism group. However, by the construction, the graphs are not co-final (see \cite{raeburn}) so that we cannot conclude whether the corresponding graph $C^{\ast}$-algebras are purely infinite or not. 
                        \end{rem}
                        {\bf Acknowledgement}: The second author is partially supported by the JC Bose National fellowship given by DST, Government of India. The third author is partially supported by ANRF/SERB MATRICS grant (Grant number MTR/2022/000515). 
                        
            \end{document}